\documentclass[11pt]{amsart}

\def\R{{\mathbb {R}}}
\def\N{{\mathbb {N}}}

\def\A{{\mathcal {A}}}

\newtheorem{teo}{Theorem}[section]
\newtheorem{lema}[teo]{Lemma}
\newtheorem{prop}[teo]{Proposition}

\theoremstyle{remark}
\newtheorem{remark}[teo]{Remark}

\theoremstyle{definition}

\numberwithin{equation}{section}

\title[concentration-compactness for variable exponent]{The concentration-compactness principle for variable exponent spaces and applications}
\author[J. Fern\'andez Bonder and A. Silva]
{Juli\'an Fern\'andez Bonder and Anal\'{\i}a Silva}

\address{Juli\'an Fern\'andez Bonder \hfill\break\indent
Departamento  de Matem\'atica, FCEyN, Universidad de Buenos Aires, \hfill\break\indent Pabell\'on I, Ciudad Universitaria (1428), Buenos Aires, Argentina.}

\email{{\tt jfbonder@dm.uba.ar}\hfill\break\indent {\it Web page:} {\tt http://mate.dm.uba.ar/$\sim$jfbonder}}

\address{Anal\'{\i}a Silva \hfill\break\indent
Departamento  de Matem\'atica, FCEyN, Universidad de Buenos Aires, \hfill\break\indent Pabell\'on I, Ciudad Universitaria (1428), Buenos Aires, Argentina.}

\email{{\tt asilva@dm.uba.ar}}

\thanks{Supported by Universidad de Buenos Aires under grant X078, by ANPCyT PICT No. 2006-290 and CONICET (Argentina) PIP 5478/1438. J. Fern\'andez Bonder is a member of CONICET. Analia Silva is a fellow of CONICET}



\begin{document}

\begin{abstract}
In this paper we extend the well-known concentration -- compactness principle of P.L. Lions to the variable exponent case. We also give some applications to the existence problem for the $p(x)-$Laplacian with critical growth.
\end{abstract}

\maketitle

\section{Introduction.}
When dealing with nonlinear elliptic equations with critical growth (in the sense of the Sobolev embeddings) the concentration -- compactness principle of P.L. Lions, see \cite{Lions}, have been proved to be a fundamental tool in order to prove existence of solutions. Just to cite a few, see \cite{Alves, Alves-Ding, Bahri-Lions, Drabek-Huang, FB, GAP} but there is an impressive list of references on this.

More recently in the analysis of some new models, that are called electrorheological fluids, the following equation has been studied,
\begin{equation}\label{pdex}
-\Delta_{p(x)} u = f(x,u) \qquad \mbox{in }\Omega.
\end{equation}
The operator $\Delta_{p(x)}u := \text{div}(|\nabla u|^{p(x)-2} \nabla u)$ is called the $p(x)-$Laplacian. When $p(x)\equiv p$ is the well-known $p-$Laplacian.

In recent years appeared a vast amount of literature that deal with the existence problem for \eqref{pdex} with different boundary conditions (Dirichlet, Neumann, nonlinear, etc). See, for instance \cite{Cabada-Pouso, Dinu, FZ, Mihailescu, Mihailescu-Radulescu} and references therein.

However, up to our knowledge, no results are available for \eqref{pdex} when the source term $f$ is allowed to have critical growth at infinity\footnote{see the remark after the introduction for more on this}. That is
$$
|f(x,t)|\le C(1+|t|^{q(x)})
$$
with $q(x)\le p^*(x):= N p(x)/(N-p(x))$ (if $p(x)<N$) and $\{q(x)=p^*(x)\}\not=\emptyset$.

This paper attempts to begin to fill this gap.

So, the objective of this paper is to extend the concentration -- compactness principle of P.L. Lions to the variable exponent setting.

The method of the proof follows the lines of the ones in the original work of P.L. Lions and the main novelty in our result is the fact that we do not require the exponent $q(x)$ to be critical everywhere. Moreover, we show that the delta masses are concentrated in the set where $q(x)$ is critical.

Finally, as an application of our result, we prove the existence of solutions to the problem
\begin{equation}\label{aplicacion 1}
\begin{cases}
-\Delta_{p(x)} u=|u|^{q(x)-2}u+\lambda(x)|u|^{r(x)-2}u & \mbox{in }\Omega\\
u=0&\mbox{ in }\partial\Omega
\end{cases}
\end{equation}
where $\Omega$ is a bounded smooth domain in $\R^N$,  $r(x)<p^*(x)-\delta$, $q(x)\leq p^*(x)$ with $\{q(x)=p^*(x)\}\not=\emptyset$.

\subsection{Statement of the results}

As we already mentioned, the main result of the paper is the extension of P.L. Lions concentration -- compactness method to the variable exponent case. More precisely, we prove,

\begin{teo}\label{ccp}
Let $q(x)$ and $p(x)$ be two continuous functions such that
$$
1<\inf_{x\in\Omega}p(x)\le \sup_{x\in\Omega}p(x) < n
\qquad \mbox{and}\qquad
1\le q(x)\le p^*(x)\quad \mbox{ in }\Omega.
$$
Let $\{u_j\}_{j\in\N}$ be a weakly convergent sequence in $W_0^{1,p(x)}(\Omega)$ with weak limit $u$, and such that:
\begin{itemize}
\item $|\nabla u_j|^{p(x)}\rightharpoonup\mu$ weakly-* in the sense of measures.

\item $|u_j|^{q(x)}\longrightarrow\nu$ weakly-* in the sense of measures.
\end{itemize}
Assume, moreover that $\A = \{x\in \Omega\colon q(x)=p^*(x)\}$ is nonempty. Then, for some countable index set $I$ we have:
\begin{align}
& \nu=|u|^{q(x)} + \sum_{i\in I}\nu_i\delta_{x_i}\quad \nu_i>0\\
& \mu \geq |\nabla u|^{p(x)} + \sum_{i\in I} \mu_i \delta_{x_i} \quad \mu_i>0\\
& S \nu_i^{1/p^*(x_i)} \leq \mu_i^{1/p(x_i)} \qquad \forall i\in I.
\end{align}
where $\{x_i\}_{i\in I}\subset \A$ and $S$ is the best constant in the Gagliardo-Nirenberg-Sobolev inequality for variable exponents, namely
$$
S = S_q(\Omega) :=\inf_{\phi\in C_0^{\infty}(\Omega)} \frac{\| |\nabla \phi| \|_{L^{p(x)}(\Omega)}}{\| \phi \|_{L^{q(x)}(\Omega)}}.
$$
\end{teo}

We want to remark that in Theorem \ref{ccp} is not required the exponent $q(x)$ to be critical {\em everywhere} and that the point masses are located in the {\em criticality set} $\A = \{x\in \Omega\colon q(x)=p^*(x)\}$.

\medskip

Now, as an application of Theorem \ref{ccp}, following the techniques of \cite{GAP} we prove the existence of solutions to
\begin{equation}\label{gap}
\begin{cases}
-\Delta_{p(x)} u = |u|^{q(x)-2}u + \lambda(x) |u|^{r(x)-2}u & \mbox{in }\Omega\\
u=0 & \mbox{on }\partial\Omega.
\end{cases}
\end{equation}

We have, in the spirit of \cite{GAP}, two types of results, depending on $r(x)$ being smaller or bigger that $p(x)$. More precisely, we prove
\begin{teo}\label{r<p}
Let $p(x)$ and $q(x)$ be as in Theorem \ref{ccp} and let $r(x)$ be continuous. Moreover, assume that $\max_{\overline\Omega}p\le \min_{\overline\Omega}q$ and $\max_{\overline\Omega}r\le \min_{\overline\Omega}p$.

Then, there exists a constant $\lambda_1>0$ depending only on $p,q,r,N$ and $\Omega$ such that if $\lambda(x)$ verifies $0<\inf_{x\in\Omega}\lambda(x)\le \|\lambda\|_{L^{\infty}(\Omega)} <\lambda_1$, then there exists infinitely many solutions to \eqref{gap} in $W^{1,p(x)}_0(\Omega)$.
\end{teo}

\begin{teo}\label{r>p}
Let $p(x)$ and $q(x)$ be as in Theorem \ref{ccp} and let $r(x)$ be continuous. Moreover, assume that $\max_{\overline{\Omega}} p \le \min_{\overline{\Omega}} r$ and that there exists $\eta>0$ such that $r(x)\le p^*(x) - \eta$ in $\Omega$.

Then, there exists $\lambda_0>0$ depending only on $p, q, r, N$ and $\Omega$, such that if
$$
\inf_{x\in A_\delta}\lambda(x)>\lambda_0 \quad \mbox{ for some } \delta>0,
$$
problem \eqref{gap} has at least one nontrivial solution in $W^{1,p(x)}_0(\Omega)$. Here, $\A_\delta$ is the $\delta-$tubular neighborhood of $\A$, namely
$$
\A_\delta := \bigcup_{x\in\A} (B_\delta(x) \cap \Omega).
$$
\end{teo}

\subsection{Organization of the paper}

After finishing this introduction, in Section 2 we give a very short overview of some properties of variable exponent Sobolev spaces that will be used throughout the paper. In Section 3 we deal with the main result of the paper. Namely the proof of the concentration -- compactness principle (Theorem \ref{ccp}). In Section 4, we begin analyzing problem \eqref{gap} and prove Theorem \ref{r>p}. Finally, in Section 5, we prove Theorem \ref{r<p}.

\subsection*{Comment on a related result}

After this paper was written, we found out that a similar result was obtained independently by Yongqiang Fu \cite{Fu}.

Even the techniques in the work of Fu are similar to the ones in this paper (and both are related to the original work by P.L. Lions), we want to remark that our results are slightly more general than those in \cite{Fu}. For instance, we do not require $q(x)$ to be critical everywhere (as is required in \cite{Fu}) and we obtain that the delta functions are located in the criticality set $\A$ (see Theorem \ref{ccp}). 

Also, in our application, again as we do not required the source term to be critical everywhere, so the result in \cite{Fu} is not applicable directly. Moreover, in Theorem \ref{r>p} our approach allows us to consider $\lambda(x)$ not necesarily a constant and the restriction that $\lambda$ is large is only needed in an $L^\infty$-norm in the criticality set.

We believe that these improvements are significant and made our result more flexible that those in \cite{Fu}.

%%%%%%%%%%%%%%%%%%%%%%%%%%%%%%%%%%%%%%%%%%
%%%%%%%%%%%%%%%%%%%%%%%%%%%%%%%%%%%%%%%%%%
%%                                      %%
%%          SECTION 2                   %%
%%                                      %%
%%%%%%%%%%%%%%%%%%%%%%%%%%%%%%%%%%%%%%%%%%
%%%%%%%%%%%%%%%%%%%%%%%%%%%%%%%%%%%%%%%%%%

\section{Results on variable exponent Sobolev spaces}

The variable exponent Lebesgue space $L^{p(x)}(\Omega)$ is defined by
$$
L^{p(x)}(\Omega) = \left\{u\in L^1_{\text{loc}}(\Omega) \colon \int_\Omega|u(x)|^{p(x)}\,dx<\infty\right\}.
$$
This space is endowed with the norm
$$
\|u\|_{L^{p(x)}(\Omega)}=\inf\left\{\lambda>0:\int_\Omega\left|\frac{u(x)}{\lambda}\right|^{p(x)}\,dx\leq 1\right\}
$$
The variable exponent Sobolev space $W^{1,p(x)}(\Omega)$ is defined by
$$
W^{1,p(x)}(\Omega) = \{u\in W^{1,1}_{\text{loc}}(\Omega) \colon u\in L^{p(x)}(\Omega) \mbox{ and } |\nabla u |\in  L^{p(x)}(\Omega)\}.
$$
The corresponding norm for this space is
$$
\|u\|_{W^{1,p(x)}(\Omega)}=\|u\|_{L^{p(x)}(\Omega)}+\| |\nabla u| \|_{L^{p(x)}(\Omega)}
$$
Define $W^{1,p(x)}_0(\Omega)$ as the closure of $C_0^\infty(\Omega)$ with respect to the $W^{1,p(x)}(\Omega)$ norm. The spaces $L^{p(x)}(\Omega)$, $W^{1,p(x)}(\Omega)$ and $W^{1,p(x)}_0(\Omega)$ are separable and reflexive Banach spaces when $1<\inf_\Omega p \le \sup_\Omega p <\infty$.

As usual, we denote $p'(x) = p(x)/(p(x)-1)$ the conjugate exponent of $p(x)$.

Define
$$
p^*(x)=\begin{cases}
\frac{Np(x)}{N-p(x)} & \mbox{ if } p(x)<N\\
\infty & \mbox{ if } p(x)\geq N
\end{cases}
$$

The following results are proved in \cite{Fan}
\begin{prop}[H\"older-type inequality]\label{Holder}
Let $f\in L^{p(x)}(\Omega)$ and $g\in L^{p'(x)}(\Omega)$. Then the following inequality holds
$$
\int_\Omega|f(x)g(x)|\,dx\leq C_p\|f\|_{L^{p(x)}(\Omega)}\|g\|_{L^{p'(x)}(\Omega)}
$$
\end{prop}

\begin{prop}[Sobolev embedding]\label{embedding}
Let $p, q\in C(\overline{\Omega})$ be such that $1\leq q(x)\le p^*(x)$ for all $x\in\overline{\Omega}$. Then there is a continuous embedding 
$$
W^{1,p(x)}(\Omega)\hookrightarrow L^{q(x)}(\Omega).
$$
Moreover, if $\inf_{\Omega} (p^*-q)>0$ then, the embedding is compact.
\end{prop}

\begin{prop}[Poincar\'e inequality]\label{Poincare}
There is a constant $C>0$, such that
$$
\|u\|_{L^{p(x)}(\Omega)}\leq C\| |\nabla u| \|_{L^{p(x)}(\Omega)},
$$
for all $u\in W^{1,p(x)}_0(\Omega)$.
\end{prop}

\begin{remark}
By Proposition \ref{Poincare}, we know that $\| |\nabla u| \|_{L^{p(x)}(\Omega)}$ and $\|u\|_{W^{1,p(x)}(\Omega)}$ are equivalent norms on $W_0^{1,p(x)}(\Omega)$.
\end{remark}

Throughout this paper the following notation will be used: Given $q\colon \Omega\to\R$ bounded, we denote
$$
q^+ := \sup_\Omega q(x), \qquad q^- := \inf_\Omega q(x).
$$

The following proposition is also proved in \cite{Fan} and it will be most usefull.
\begin{prop}\label{norma.y.rho}
Set $\rho(u):=\int_\Omega|u(x)|^{p(x)}\,dx$. For $u,\in L^{p(x)}(\Omega)$ and $\{u_k\}_{k\in\N}\subset L^{p(x)}(\Omega)$, we have
\begin{align}
& u\neq 0 \Rightarrow \Big(\|u\|_{L^{p(x)}(\Omega)} = \lambda \Leftrightarrow \rho(\frac{u}{\lambda})=1\Big).\\
& \|u\|_{L^{p(x)}(\Omega)}<1 (=1; >1) \Leftrightarrow \rho(u)<1(=1;>1).\\
& \|u\|_{L^{p(x)}(\Omega)}>1 \Rightarrow \|u\|^{p^-}_{L^{p(x)}(\Omega)} \leq \rho(u) \leq \|u\|^{p^+}_{L^{p(x)}(\Omega)}.\\
& \|u\|_{L^{p(x)}(\Omega)}<1 \Rightarrow \|u\|^{p^+}_{L^{p(x)}(\Omega)} \leq \rho(u) \leq \|u\|^{p^-}_{L^{p(x)}(\Omega)}.\\
& \lim_{k\to\infty}\|u_k\|_{L^{p(x)}(\Omega)} = 0 \Leftrightarrow \lim_{k\to\infty}\rho(u_k)=0.\\
& \lim_{k\to\infty}\|u_k\|_{L^{p(x)}(\Omega)} = \infty \Leftrightarrow \lim_{k\to\infty}\rho(u_k) = \infty.
\end{align}
\end{prop}
%%%%%%%%%%%%%%%%%%%%%%%%%%%%%%%%%%%%%%%%%%
%%%%%%%%%%%%%%%%%%%%%%%%%%%%%%%%%%%%%%%%%%
%%                                      %%
%%          SECTION 3                   %%
%%                                      %%
%%%%%%%%%%%%%%%%%%%%%%%%%%%%%%%%%%%%%%%%%%
%%%%%%%%%%%%%%%%%%%%%%%%%%%%%%%%%%%%%%%%%%

\section{concentration compactness principle}
Let $\{u_j\}_{j\in\N}$ be a bounded sequence in $W_0^{1,p(x)}(\Omega)$ and let $q\in C(\overline{\Omega})$ be such that $q\le p^*$ with $\{x\in\Omega\colon q(x)=p^*(x)\} \not= \emptyset$. Then there exists a subsequence that we still denote by $\{u_j\}_{j\in\N}$, such that
\begin{itemize}
\item $u_j\rightharpoonup u\quad \mbox{ weakly in } W_0^{1,p(x)}(\Omega)$,
\item $u_j\to u \quad \mbox{ strongly in } L^{r(x)}(\Omega)\quad\forall1\leq r(x)<p^*(x)$,
\item $|u_j|^{q(x)}\rightharpoonup \nu$ weakly * in the sense of measures,
\item $|\nabla u_j|^{p(x)}\rightharpoonup \mu$ weakly * in the sense of measures.
\end{itemize}

Consider $\phi\in C^\infty(\overline{\Omega})$, from the Poincar\'e inequality for variable exponents, we obtain
\begin{equation}\label{poincare}
\|\phi u_j\|_{L^{q(x)}(\Omega)} S\leq\|\nabla(\phi u_j)\|_{L^{p(x)}(\Omega)}.
\end{equation}
On the other hand,
$$
| \|\nabla(\phi u_j)\|_{L^{p(x)}(\Omega)}-\|\phi\nabla u_j\|_{L^{p(x)}(\Omega)}|\le \| u_j \nabla \phi\|_{L^{p(x)}(\Omega)}.
$$
We first assume that $\nu=0$. Then, we observe that the right side of the inequality converges to 0. In fact, if, for instance $\| |u|^{p(x)}\|_{L^1(\Omega)}\ge 1$,
\begin{align*}
\| u_j \nabla \phi\|_{L^{p(x)}(\Omega)}&\le (\|\nabla \phi\|_{L^{\infty}(\Omega)}+1)^{p^+} \| u_j \|_{L^{p(x)}(\Omega)}\\
& \le (\|\nabla \phi\|_{L^{\infty}(\Omega)}+1)^{p^+} \| |u|^{p(x)}\|_{L^1(\Omega)}^{1/p_-}\to 0
\end{align*}

Finally, if we take the limit for $j\to \infty$ in \eqref{poincare}, we have,
\begin{equation}\label{RH}
\|\phi\|_{L_\nu^{q(x)}(\Omega)} S\leq\|\phi\|_{L_\mu^{p(x)}(\Omega)}
\end{equation}

Now we need a lemma that is the key role in the proof of Theorem \ref{ccp}.

\begin{lema}\label{Lema 1}
Let $\mu,\nu$ be two non-negative and bounded measures on $\overline{\Omega}$, such that for $1\leq p(x)<r(x)<\infty$ there exists some  constant $C>0$ such that
$$
\|\phi\|_{L_\nu^{r(x)}(\Omega)}\leq C\|\phi\|_{L_\mu^{p(x)}(\Omega)}
$$
Then, there exist $\{x_j\}_{j\in J}\subset\overline{\Omega}$ and $\{\nu_j\}_{j\in I}\subset (0,\infty)$,such that:
$$
\nu=\Sigma\nu_i\delta_{x_i}
$$
\end{lema}

For the proof of Lemma \ref{Lema 1} we need a couple of preliminary results.

\begin{lema}\label{Lema 2}
Let $\nu$ be a non-negative bounded measure. Assume that there exists $\delta>0$ such that for all $A$ borelian, $\nu(A)=0$ or $\nu(A)\geq\delta$. Then, there exist $\{x_i\}$ and $\nu_i>0$ such that
$$
\nu=\sum \nu_i\delta_{x_i}
$$
\end{lema}

\begin{proof}
The proof is elementary and is left to the reader.
\end{proof}

\begin{lema}\label{Lema 3}
Let $\nu$ be non-negative and bounded measures,such that
$$
\|\psi\|_{L_\nu^{r(x)}(\Omega)}\leq C\|\psi\|_{L_\nu^{p(x)}(\Omega)}
$$
Then there exist $\delta>0$ such that for all $A$ borelian, $\nu(A)=0$ or $\nu(A)\geq\delta$.
\end{lema}

\begin{proof}
First, observe that if $\nu(A)\geq1$,
$$
\int_\Omega\left(\frac{\chi_{A}(x)}{\nu(A)^\frac{1}{p-}}\right)^{p(x)}\, d\nu\leq \int_\Omega \left(\frac{\chi_{A}(x)} {\nu(A)^\frac{1}{p(x)}}\right)^{p(x)}\, d\nu = 1.
$$
Then $\nu(A)^\frac{1}{p-}\geq\|\chi_{A}\|_{L_\nu^{p(x)}}$.
On the other hand,
$$
\int_\Omega \left(\frac{\chi_{A}(x)} {\nu(A)^\frac{1}{r+}}\right)^{r(x)}\, d\nu \geq \int_\Omega\frac{\chi_A(x)}{\nu(A)}\, d\nu = 1.
$$
Then $\nu(A)^\frac{1}{r+}\geq\|\chi_{A}\|_{L_\nu^{r(x)}}$. So we conclude that
$$
\nu(A)^\frac{1}{r+}\leq C\nu(A)^\frac{1}{p-}.
$$

Now, if $\nu(A)<1$, we obtain that
$$
\nu(A)^\frac{1}{r-}\leq C\nu(A)^\frac{1}{p+}.
$$

Combining all these facts, we arrive at
$$
\min\left\{\nu(A)^\frac{1}{r-},\nu(A)^\frac{1}{r+}\right\}\leq C \max\left\{\nu(A)^\frac{1}{p-}, \nu(A)^\frac{1}{p+}\right\}.
$$

Now, if $\nu(A)\leq1$, we have that
$$
\nu(A)^\frac{1}{r-}\leq C\nu(A)^\frac{1}{p+}
$$
Then, $\nu(A)=0$ or
$$
\nu(A)\geq(\frac{1}{C})^{\frac{p^+r^-}{r^- -p^+}}
$$
Finally,
$$
\nu(A)\geq\min\left\{(\frac{1}{C})^{\frac{p^+r^-}{r^- -p^+}},1\right\}
$$
This finishes the proof.
\end{proof}

Now we are ready to prove Lemma \ref{Lema 1}

\begin{proof}[\bf Proof of Lemma \ref{Lema 1}]
By reverse H\"older inequality \eqref{RH}, the measure $\nu$ is absolutely continuous with respect to $\mu$. As consequence there exists $f\in L_\mu^1(\Omega)$, $f\geq0$, such that $\nu=\mu\lfloor f$. Also by \eqref{RH} we have,
$$
\min\left\{\nu(A)^\frac{1}{r-},\nu(A)^\frac{1}{r+}\right\}\leq C\max\left\{\mu(A)^\frac{1}{p-},\mu(A)^\frac{1}{p+}\right\}
$$
for any Borel set $A\subset\Omega$. In particular, $f\in L^\infty_\mu(\Omega)$.
On the other hand the Lebesgue decomposition of $\mu$ with respect to $\nu$ gives us
$$
\mu=\nu\lfloor g + \sigma,\mbox{ where } g\in L^1_\nu(\Omega), g\geq0
$$
and $\sigma$ is a bounded positive measure, singular with respect to $\nu$.

Now consider \eqref{RH} applied to the test function
$$
\phi=g^\frac{1}{r(x)-p(x)}\chi_{\{g\leq n\}}\psi.
$$
We obtain
\begin{align*}
\|g^\frac{1}{r(x)-p(x)}\chi_{\{g\leq n\}}\psi\|_{L^{r(x)}_\nu}&\leq C\|g^\frac{1}{r(x)-p(x)}\chi_{\{g\leq n\}}\psi\|_{L^{p(x)}_\mu}\\
&= C\|g^\frac{1}{r(x)-p(x)}\chi_{\{g\leq n\}}\psi\|_{L^{p(x)}_{g d\nu+d\sigma}}\\
&\leq C\|g^\frac{r(x)}{p(x)(r(x)-p(x))} \chi_{\{g\leq n\}}\psi\|_{L^{p(x)}_\nu} + C\|g^\frac{1}{r(x)-p(x)}\chi_{\{g\leq n\}}\psi\|_{L^{p(x)}_\sigma}
\end{align*}
Since $\sigma\perp\nu$, we have:
$$
\|g^\frac{1}{r(x)-p(x)}\chi_{\{g\leq n\}}\psi\|_{L^{r(x)}_\nu}\leq C\|g^\frac{r(x)}{p(x)(r(x)-p(x))}\chi_{\{g\leq n\}}\psi\|_{L^{p(x)}_\nu}
$$
Hence calling $d\nu_n=g^\frac{r(x)}{(r(x)-p(x))}\chi_{g\leq n} d\nu$ the following reverse H\"older inequality holds.
$$
\|\psi\|_{L^{r(x)}_{\nu_n}}\leq\|\psi\|_{L^{p(x)}_{\nu_n}}
$$
By Lemma \ref{Lema 2} and Lemma \ref{Lema 3}, there exists $x_i^n$ and $K_i^n>0$ such that $\nu_n = \sum_{i\in I} K_i^n\delta_{x_i}^n$.
On the other hand, $\nu_n\nearrow g^\frac{r(x)}{r(x)-p(x)}\nu$. Then, we have
$$
g^\frac{r(x)}{r(x)-p(x)}\nu = \sum_{i\in I} K_i^n\delta_{x_i}^n
$$
where $K_i=g^\frac{r(x_i)}{r(x_i)-p(x_i)}(x_i)\nu(x_i)$.
This finishes the proof.
\end{proof}

The following Lemma follows exactly as in the constant exponent case and the proof is omitted.

\begin{lema}\label{lema 4}
Let $f_n\to f$ a.e and $f_n\rightharpoonup f$ in $L^{p(x)}(\Omega)$ then
$$
\lim_{n\to\infty}\left(\int_\Omega|f_n|^{p(x)}dx-\int_\Omega|f-f_n|^{p(x)}dx\right)=\int_\Omega|f|^{p(x)}dx
$$
\end{lema}
%\begin{proof}
%We define
%$$
%W_{\varepsilon,n}(x)=\left(||f_n(x)|^{p(x)}-|f(x)-f_n(x)|^{p(x)}-|f(x)|^{p(x)}|-\varepsilon|f_n(x)|^{p(x)}\right)_+
%$$
%It is evident that $W_{\varepsilon,n}(x)\longrightarrow0 $ a.e.
%On the other hand,by
%$$
%\left||a+b|^{p(x)}-|a|^{p(x)}\right|\leq\varepsilon|a|^{p(x)}+C_{\varepsilon}|b|^{p(x)}
%$$
%We obtain that,
%$$
%\left||f_n(x)|^{p(x)}-|f(x)-f_n(x)|^{p(x)}-|f(x)|^{p(x)}\right|\leq\varepsilon|f_n(x)|^{p(x)}+(C_\varepsilon+1)|f(x)|^{p(x)}
%$$
%It is immediate that $\int_\Omega W_{\varepsilon,n}(x)\,dx\longrightarrow0$.
%Then, we namely
%$$
%I_n=\int_\Omega||f_n(x)|^{p(x)}-|f(x)-f_n(x)|^{p(x)}-|f(x)|^{p(x)}| dx
%$$
%we have,
%$$
%I_n\leq\int_\Omega W_{\varepsilon,n}(x) dx+C\varepsilon
%$$
%It is clear that $\limsup I_n\leq C\varepsilon$ for all $\varepsilon>0$. This finished the proof.
%\end{proof}

Now we are in position to prove Theorem \ref{ccp}.

\begin{proof}[\bf Proof of Theorem \ref{ccp}]
Given any $\phi\in C^\infty (\Omega)$ we write $v_j=u_j-u$ and by Lemma \ref{lema 4}, we have
$$
\lim_{j\to\infty}\left(\int_\Omega |\phi|^{q(x)}|u_j|^{q(x)}-\int_\Omega|\phi|^{q(x)}|v_j|^{q(x)} dx\right)=\int_\Omega |\phi|^{q(x)}|u|^{q(x)} dx.
$$
On the other hand, by reverse H\"older inequality \eqref{RH} and Lemma \ref{Lema 1}, taking limits we obtain the representation
$$
\nu = |u|^{q(x)} + \sum_{j\in I} \nu_j\delta_{x_j}
$$

\medskip

Let us now show that the points $x_j$ actually belong to the {\em critical set} $\A$.

In fact, assume by contradiction that $x_1\in \Omega\setminus \A$. Let $B=B(x_1,r) \subset\subset \Omega-\A$. Then $q(x)<p^*(x)-\delta$ for some $\delta>0$ in $\overline{B}$ and, by Proposition \ref{embedding}, The embedding $W^{1,p(x)}(B)\hookrightarrow L^{q(x)}(B)$ is compact. Therefore, $u_j\to u$ strongly in $L^{q(x)}(B)$ and so $|u_j|^{q(x)}\to |u|^{q(x)}$ strongly in $L^1(B)$. This is a contradiction to our assumption that $x_1\in B$.

\medskip

Now we proceed with the proof.

Applying \eqref{poincare} to $\phi u_j$ and taking into account that $u_j\to u$ in $L^{p(x)}(\Omega)$, we have
$$
\|\phi\|_{L^{q(x)}_\nu(\Omega)} \leq \|\phi\|_{L^{p(x)}_\mu(\Omega)} + \|(\nabla \phi) u\|_{L^{p(x)}(\Omega)}
$$
Consider $\phi\in C^\infty_c(\R^n)$ such that $0\leq\phi\leq1$, $\phi(0)=1$ and supported in the unit ball of $\R^n$. Fixed $j\in I$, we consider $\varepsilon>0$ such that $B_{\varepsilon}(x_i)\cap B_{\varepsilon}(x_j)=\emptyset$ for $i\neq j$.

We denote by $\phi_{\varepsilon,j}(x):= \varepsilon^{-n}\phi((x-x_j)/\varepsilon)$.

By decomposition of $\nu$, we have:
\begin{align*}
\rho_\nu(\phi_{i_0,\varepsilon}) &:= \int_\Omega|\phi_{i_0,\varepsilon}|^{q(x)}\,d\nu \\
&= \int_\Omega |\phi_{i_0,\varepsilon}|^{q(x)}|u|^{q(x)}\, dx + \sum_{i\in I} \nu_i\phi_{i_0,\varepsilon}(x_i)^{q(x_i)}\\
&\geq \nu_{i_0}.
\end{align*}

From now on, we will denote
\begin{align*}
q^+_{i,\varepsilon}:=\sup_{B_\varepsilon(x_i)}q(x),\qquad q^-_{i,\varepsilon}:=\inf_{B_\varepsilon(x_i)}q(x),\\
p^+_{i,\varepsilon}:=\sup_{B_\varepsilon(x_i)}p(x),\qquad p^-_{i,\varepsilon}:=\inf_{B_\varepsilon(x_i)}p(x).
\end{align*}

If $\rho_\nu(\phi_{i_0,\varepsilon})<1$ then
$$
\|\phi_{i_0,\varepsilon}\|_{L^{q(x)}_\nu(\Omega)} = \|\phi_{i_0,\varepsilon}\|_{L^{q(x)}_\nu (B_\varepsilon(x_{i_0}))} \ge \rho_\nu(\phi_{i_0,\varepsilon})^{1/q^-_{i,\varepsilon}} \ge \nu_{i_0}^{1/q^-_{i,\varepsilon}}.
$$
Analogously, if $\rho_\nu(\phi_{i_0,\varepsilon})>1$ then
$$
\|\phi_{i_0,\varepsilon}\|_{L^{q(x)}_\nu(\Omega)} \ge \nu_{i_0}^{1/q^+_{i,\varepsilon}}.
$$
Then,
$$
\min\left\{\nu_i^\frac{1}{q^+_{i,\varepsilon}}, \nu_i^\frac{1}{q^-_{i,\varepsilon}}\right\} S \leq \|\phi_{i,\varepsilon}\|_{L^{p(x)}_\mu(\Omega)} + \|(\nabla\phi_{i,\varepsilon}) u\|_{L^{p(x)}(\Omega)}.
$$

Now, by Proposition \ref{norma.y.rho} we have
$$\|(\nabla\phi_{i,\varepsilon}) u\|_{L^{p(x)}(\Omega)} \le \max\{\rho((\nabla\phi_{i,\varepsilon}) u)^{1/p^-}; \rho((\nabla\phi_{i,\varepsilon}) u)^{1/p^+}\}.
$$
Now, by H\"older inequality we have
\begin{align*}
\rho((\nabla\phi_{i,\varepsilon}) u)
& =  \int_{\Omega} |\nabla\phi_{i,\varepsilon}|^{p(x)} |u|^{p(x)}\, dx\\
& \leq \| |u|^{p(x)}\|_{L^{\alpha(x)}(B_\varepsilon(x_i))} \||\nabla \phi_{i,\varepsilon}|^{p(x)} \|_{L^{\alpha'(x)}(B_\varepsilon(x_i))},
\end{align*}
where $\alpha(x) = n/(n-p(x))$ and $\alpha'(x) = n/p(x)$.

Moreover, using that $\nabla\phi_{i,\varepsilon}=\nabla\phi\left(\frac{x-x_i}{\varepsilon}\right)\frac{1}{\varepsilon}$, we obtain:
$$
\||\nabla \phi_{i,\varepsilon}|^{p(x)} \|_{L^{\alpha'(x)}(B_\varepsilon(x_i))}\le \max\{\rho(|\nabla \phi_{i,\varepsilon}|^{p(x)})^{p^+/n};
\rho(|\nabla \phi_{i,\varepsilon}|^{p(x)})^{p^-/n}\},
$$
and
\begin{align*}
\rho(|\nabla \phi_{i,\varepsilon}|^{p(x)})
& = \int_{B_\varepsilon(x_i)} |\nabla\phi_{i,\varepsilon}|^n\, dx\\
& = \int_{B_\varepsilon(x_i)} |\nabla\phi(\frac{x-x_i}{\varepsilon})|^n \frac{1}{\varepsilon^n}\, dx \\
& = \int_{B_1(0)} |\nabla\phi(y)|^n\, dy.
\end{align*}
Then, $\nabla\phi_{i,\varepsilon} u\to 0$ strongly in $L^{p(x)}(\Omega)$.
On the other hand,
$$
\int_\Omega |\phi_{i,\varepsilon}|^{p(x)}\,d\mu\leq \mu(B_{\varepsilon}(x_i))=\mu_i
$$
Therefore,
\begin{align*}
\|\phi_{i,\varepsilon}\|_{L^{p(x)}(\Omega)} &= \|\phi_{i,\varepsilon}\|_{L^{p(x)}(B_\varepsilon(x_i))} \\
&\leq  \max \left\{ \rho_\mu (\phi_{i,\varepsilon})^\frac{1}{p^+_{i,\varepsilon}}, \rho_\mu(\phi_{i,\varepsilon})^\frac{1} {p^-_{i,\varepsilon}}\right\}\\
&\le \max \left\{ \mu_i^\frac{1}{p^+_{i,\varepsilon}}, \mu_i^\frac{1}{p^-_{i,\varepsilon}}\right\},
\end{align*}
so we obtain,
$$
S\min \left\{ \nu_i^\frac{1}{q^+_{i,\varepsilon}}, \nu_i^\frac{1}{q^-_{i,\varepsilon}}\right\} \leq \max \left\{ \mu_i^\frac{1}{p^+_{i,\varepsilon}}, \mu_i^\frac{1}{p^-_{i,\varepsilon}}\right\}.
$$
As $p$ and $q$ are continuous functions and as $q(x_i) = p^*(x_i)$, letting $\varepsilon\to 0$, we get
$$
S \nu_i^{1/p^*(x_i)} \le \mu_i^{1/p(x_i)}
$$

\medskip

Finally, we show that $\mu\geq|\nabla u|^{p(x)}+\Sigma\mu_i\delta_{x_i}$.

In fact, we have that $\mu\geq\sum\mu_i\delta_{x_i}$.
On the other hand $u_j\rightharpoonup u$ weakly in $W_0^{1,p(x)}(\Omega)$ then $\nabla u_j\rightharpoonup\nabla u$ weakly in $L^{p(x)}(U)$ for all $U\subset\Omega$. By weakly lower semi continuity  of norm we obtain that $d\mu\geq|\nabla u|^{p(x)}\,dx$ and, as $|\nabla u|^{p(x)}$ is orthogonal to $\mu_1$, we conclude the desired result.

This finishes the proof.
\end{proof}

\section{Applications}

In this section, we apply Theorem \ref{ccp} to study the existence of nontrivial solutions of the problem
\begin{equation}\label{aplicacion}
\begin{cases}
-\Delta_{p(x)} u=|u|^{q(x)-2}u+\lambda(x)|u|^{r(x)-2}u & \mbox{in }\Omega,\\
u=0&\mbox{ in }\partial\Omega,
\end{cases}
\end{equation}
where $r(x)<p^*(x)-\varepsilon$, $q(x)\leq p^*(x)$ and $\A = \{x\in\Omega \colon q(x)=p^*(x)\}\neq\emptyset$.
We define $A_\delta := \bigcup_{x\in \A} (B_\delta(x)\cap\Omega) = \{x\in \Omega\colon \mbox{dist}(x,\A)<\delta\}$.

The ideas for this application follow those in the paper \cite{GAP}.

\medskip

For (weak) solutions of \eqref{aplicacion} we understand critical points of the functional
$$
\mathcal{F}(u)=\int_\Omega\frac{|\nabla u|^{p(x)}}{p(x)}-\frac{|u|^{q(x)}}{q(x)}-\lambda(x)\frac{|u|^{r(x)}}{r(x)}\,dx
$$

\subsection{Proof of Theorem \ref{r>p}}
We begin by proving the Palais-Smale condition for the functional $\mathcal{F}$, below certain level of energy.

\begin{lema}\label{acotada}
Assume that $r\le q$. Let $\{u_j\}_{j\in\N}\subset W_0^{1,p(x)}(\Omega)$ a Palais-Smale sequence then $\{u_j\}_{j\in\N}$ is bounded in $W_0^{1,p(x)}(\Omega)$.
\end{lema}
\begin{proof}
By definition
$$
\mathcal{F}(u_j)\to c \qquad \mbox{and}\qquad \mathcal{F}'(u_j)\to 0.
$$
Now, we have
$$
c+1 \geq \mathcal{F}(u_j) = \mathcal{F}(u_j) - \frac{1}{r-} \langle \mathcal{F}'(u_j), u_j \rangle + \frac{1}{r-} \langle \mathcal{F}'(u_j), u_j \rangle,
$$
where
$$
\langle \mathcal{F}'(u_j), u_j \rangle = \int_\Omega |\nabla u_j|^{p(x)} - |u_j|^{q(x)} - |u_j|^{r(x)}\, dx.
$$
Then, if $r(x)\leq q(x)$ we conclude
$$
c+1 \geq \left(\frac{1}{p+} - \frac{1}{r-}\right) \int_\Omega |\nabla u_j|^{p(x)}\, dx - \frac{1}{r-} |\langle \mathcal{F}'(u_j), u_j \rangle|.
$$
We can assume that $\|u_j\|_{W_0^{1,p(x)}(\Omega)}\geq 1$. As $\|\mathcal{F}'(u_j)\|$ is bounded we have that
$$
c+1 \geq \left(\frac{1}{p+} - \frac{1}{r-}\right) \|u_j\|^{p^-}_{W_0^{1,p(x)}(\Omega)} - \frac{C}{r-}\|u_j\|_{W_0^{1,p(x)}(\Omega)}.
$$
We deduce that $u_j$ is bounded.

This finishes the proof.
\end{proof}

From the fact that $\{u_j\}_{j\in\N}$ is a Palais-Smale sequence it follows, by Lemma \ref{acotada}, that $\{u_j\}_{j\in\N}$ is bounded in $W_0^{1,p(x)}(\Omega)$. Hence, by Theorem \ref{ccp}, we have
\begin{align}
&|u_j|^{q(x)}\rightharpoonup \nu=|u|^{q(x)} + \sum_{i\in I} \nu_i\delta_{x_i} \quad \nu_i>0,\\
&|\nabla u_j|^{p(x)}\rightharpoonup \mu \geq |\nabla u|^{p(x)} + \sum_{i\in I} \mu_i \delta_{x_i}\quad \mu_i>0,\\
&S \nu_i^{1/p^*(x_i)} \leq \mu_i^{1/p(x_i)}.
\end{align}

Note that if $I=\emptyset$ then $u_j\to u$ strongly in $L^{q(x)}(\Omega)$. We know that $\{x_i\}_{i\in I}\subset \A$.

Let us show that if $c < \left(\frac{1}{p^+} - \frac{1}{q^-_\A}\right)S^n$ and $\{u_j\}_{j\in\N}$ is a Palais-Smale sequence, with energy level $c$, then $I=\emptyset$.

\medskip

In fact, suppose that $I \not= \emptyset$. Then let $\phi\in C_0^\infty(\R^n)$ with support in the unit ball of $\R^n$. Consider, as in the previous section, the rescaled functions $\phi_{i,\varepsilon}(x) = \phi(\frac{x-x_i}{\varepsilon})$.

As $\mathcal{F}'(u_j)\to 0$ in $(W_0^{1,p(x)}(\Omega))'$, we obtain that
$$
\lim_{j\to\infty} \langle \mathcal{F}'(u_j), \phi_{i,\varepsilon}u_j \rangle = 0.
$$
On the other hand,
$$
\langle \mathcal{F}'(u_j), \phi_{i,\varepsilon} u_j \rangle = \int_\Omega |\nabla u_j|^{p(x)-2}\nabla u_j\nabla(\phi_{i,\varepsilon}u_j) - \lambda(x) |u_j|^{r(x)}\phi_{i,\varepsilon} - |u_j|^{q(x)}\phi_{i,\varepsilon}\, dx
$$
Then, passing to the limit as $j\to\infty$, we get
\begin{align*}
0 = &\lim_{j\to\infty} \left(\int_\Omega |\nabla u_j|^{p(x)-2} \nabla u_j \nabla(\phi_{i,\varepsilon}) u_j\, dx\right)\\
& + \int_\Omega \phi_{i,\varepsilon}\, d\mu - \int_\Omega \phi_{i,\varepsilon}\, d\nu
- \int_\Omega\lambda(x) |u|^{r(x)}\phi_{i,\varepsilon}\, dx.
\end{align*}
By H\"older inequality, it is easy to check that
$$
\lim_{j\to\infty} \int_\Omega |\nabla u_j|^{p(x)-2} \nabla u_j \nabla(\phi_{i,\varepsilon}) u_j\, dx = 0.
$$
On the other hand,
$$
\lim_{\varepsilon\to 0} \int_\Omega \phi_{i,\varepsilon}\, d\mu = \mu_i\phi(0),\qquad
\lim_{\varepsilon\to 0} \int_\Omega \phi_{i,\varepsilon}\, d\nu = \nu_i\phi(0).
$$
and
$$
\lim_{\varepsilon\to 0} \int_\Omega \lambda(x)|u|^{r(x)}\phi_{i,\varepsilon}\, dx = 0.
$$
So, we conclude that $(\mu_i-\nu_i)\phi(0)=0$, i.e, $\mu_i=\nu_i$. Then,
$$
S \nu_i^{1/p^*(x_i)} \leq \nu_i^{1/p(x_i)},
$$
so it is clear that $\nu_i = 0$ or $S^n\leq\nu_i$.

On the other hand, as $r^->p^+$,
\begin{align*}
c =& \lim_{j\to\infty} \mathcal{F}(u_j) = \lim_{j\to\infty} \mathcal{F}(u_j) - \frac{1}{p+} \langle \mathcal{F}'(u_j), u_j \rangle\\
=& \lim_{j\to\infty} \int_\Omega \left(\frac{1}{p(x)} - \frac{1}{p+}\right) |\nabla u_j|^{p(x)}\, dx + \int_\Omega \left(\frac{1}{p+} - \frac{1}{q(x)}\right) |u_j|^{q(x)}\, dx\\
& + \lambda \int_\Omega \left(\frac{1}{p+}-\frac{1}{r(x)}\right)| u_j|^{r(x)}\, dx\\
\geq& \lim_{j\to\infty} \int_\Omega \left(\frac{1}{p+}-\frac{1}{q(x)}\right) |u_j|^{q(x)}\, dx\\
\geq& \lim_{j\to\infty} \int_{\A_\delta} \left(\frac{1}{p+}-\frac{1}{q(x)}\right) |u_j|^{q(x)}\, dx\\
\geq&\lim_{j\to\infty} \int_{\A_\delta} \left(\frac{1}{p+}-\frac{1}{q^-_{\A_\delta}}\right) |u_j|^{q(x)}\, dx
\end{align*}
But
\begin{align*}
\lim_{j\to\infty} \int_{\A_\delta} \left(\frac{1}{p+} - \frac{1}{q^-_{\A_\delta}}\right) |u_j|^{q(x)}\, dx &= \left(\frac{1}{p+}-\frac{1}{q^-_{\A_\delta}}\right) \left(\int_{\A_\delta}|u|^{q(x)}\, dx + \sum_{j\in I} \nu_j\right)\\
&\geq \left(\frac{1}{p+} - \frac{1}{q^-_{\A_\delta}}\right) \nu_i\\
&\geq \left(\frac{1}{p+}-\frac{1}{q^-_{\A_\delta}}\right) S^n.
\end{align*}
As $\delta>0$ is arbitrary, and $q$ is continuous, we get
$$
c\ge \left(\frac{1}{p+}-\frac{1}{q^-_{\A}}\right) S^n.
$$

Therefore, if
$$
c < \left(\frac{1}{p+} - \frac{1}{q^-_{\A}}\right)S^n,
$$
the index set $I$ is empty.

Now we are ready to prove the Palais-Smale condition below level $c$.
\begin{teo}\label{Lemma.PS}
Let $\{u_j\}_{j\in\N}\subset W_0^{1,p(x)}(\Omega)$ be a  Palais-Smale sequence, with energy level $c$. If $c < \left(\frac{1}{p+} - \frac{1}{q^-_{\A}}\right) S^n$, then there exist $u\in W_0^{1,p(x)}(\Omega)$ and $\{u_{j_k}\}_{k\in\N}\subset \{u_j\}_{j\in\N}$ a subsequence such that $u_{j_{k}}\to u$ strongly in $W_0^{1,p(x)}(\Omega)$.
\end{teo}

\begin{proof}
We have that $\{u_j\}_{j\in\N}$ is bounded. Then, for a subsequence that we still denote $\{u_j\}_{j\in\N}$, $u_j\to u$ strongly in $L^{q(x)}(\Omega)$. We define $\mathcal{F}'(u_j):=\phi_j$. By the Palais-Smale condition, with energy level c, we have $\phi_j\to 0$ in $(W_0^{1,p(x)}(\Omega))'$.

By definition $\langle \mathcal{F}'(u_j), z \rangle = \langle \phi_j, z \rangle$ for all $z\in W_0^{1,p(x)}(\Omega)$, i.e,
$$
\int_\Omega |\nabla u_j|^{p(x)-2}\nabla u_j\nabla z\, dx - \int_\Omega |u_j|^{q(x)-2} u_j z\, dx - \int_\Omega\lambda(x) |u_j|^{r(x)-2} u_j z\, dx = \langle \phi_j, z \rangle.
$$
Then, $u_j$ is a weak solution of the following equation.
\begin{equation}
\begin{cases}
-\Delta_{p(x)}u_j=|u_j|^{q(x)-2}u_j+\lambda(x)|u_j|^{r(x)-2}u_j+\phi_j=: f_j &\mbox{in }\Omega,\\
u_j = 0 &\mbox{on }\partial\Omega.
\end{cases}
\end{equation}
We define $T\colon (W_0^{1,p(x)}(\Omega))' \to W_0^{1,p(x)}(\Omega)$, $T(f):=u$ where $u$ is the weak solution of the following equation.
\begin{equation}
\begin{cases}
-\Delta_{p(x)}u=f & \mbox{in } \Omega,\\
u = 0 & \mbox{on } \partial\Omega.
\end{cases}
\end{equation}
Then $T$ is a continuous invertible operator.

It is sufficient to show that $f_j$ converges in $(W^{1,p(x)}_0(\Omega))'$. We only need to prove that $|u_j|^{q(x)-2}u_j \to |u|^{q(x)-2}u$ strongly in $(W_0^{1,p(x)}(\Omega))'$.

In fact,
\begin{align*}
\langle|u_j|^{q(x)-2}u_j-|u|^{q(x)-2}u,\psi\rangle&=\int_\Omega(|u_j|^{q(x)-2}u_j-|u|^{q(x)-2}u)\psi\,dx\\
&\leq\|\psi\|_{L^{q(x)}(\Omega)}\|(|u_j|^{q(x)-2}u_j-|u|^{q(x)-2}u)\|_{L^{q'(x)}(\Omega)}.
\end{align*}
Therefore,
\begin{align*}
\|(|u_j|^{q(x)-2}u_j - |u|^{q(x)-2}u)\|_{(W_0^{1,p(x)}(\Omega))'} &= \sup_{\genfrac{}{}{0cm}{}{\psi\in W^{1,p(x)}_0(\Omega)} {\|\psi\|_{W^{1,p(x)}_0(\Omega)}=1}} \int_\Omega (|u_j|^{q(x)-2}u_j-|u|^{q(x)-2}u)\psi\, dx\\
&\leq \|(|u_j|^{q(x)-2}u_j - |u|^{q(x)-2}u)\|_{L^{q'(x)}(\Omega)}
\end{align*}
and now, by the Dominated Convergence Theorem this last term goes to zero as $j\to\infty$.

The proof is finished.
\end{proof}

We are now in position to prove Theorem \ref{r>p}.

\begin{proof}[\bf Proof of Theorem \ref{r>p}]
In view of the previous result, we seek for critical values below level $c$. For that purpose, we want to use the Mountain Pass Theorem. Hence we have to check the following condition:
\begin{enumerate}
\item There exist constants $R,r>0$ such that when $\|u\|_{W^{1,p(x)}(\Omega)}=R$, then $\mathcal{F}(u)>r$.

\item There exist $v_0\in W^{1,p(x)}(\Omega)$ such that $\mathcal{F}(v_0)<r$.
\end{enumerate}
Let us first check (1). We suppose that $\| |\nabla u| \|_{L^{p(x)}(\Omega)}\leq 1$ and $\| u\|_{L^{p(x)}(\Omega)}\leq 1$. The other cases can be treated similarly.

By Poincar\'e inequality (Proposition \ref{poincare}) we have,
\begin{align*}
\int_\Omega \frac{|\nabla u|^{p(x)}}{p(x)} - &\frac{|u|^{q(x)}}{q(x)} - \lambda(x)\frac{|u|^{r(x)}}{r(x)}\, dx\\
&\geq \frac{1}{p+} \int_\Omega|\nabla u|^{p(x)}\, dx - \frac{1}{q-} \int_\Omega|u|^{q(x)}\, dx - \frac{\|\lambda\|_\infty}{r-} \int_\Omega|u|^{r(x)}\, dx\\
&\geq \frac{1}{p+} \| |\nabla u| \|^{p+} - \frac{1}{q-} \|u\|_{L^{q(x)}(\Omega)}^{q-} - \frac{\|\lambda\|_\infty}{r-} \|u\|_{L^{r(x)}(\Omega)}^{r-}\\
&\geq \frac{1}{p+} \| |\nabla u| \|^{p+} - \frac{C}{q-}\| |\nabla u| \|_{L^{p(x)}(\Omega)}^{q-} - \frac{C\|\lambda\|_\infty}{r-} \| |\nabla u| \|_{L^{p(x)}(\Omega)}^{r-}.
\end{align*}
Let $g(t) = \frac{1}{p+} t^{p+} - \frac{C}{q-} t^{q-} - \frac{C\|\lambda\|_\infty}{r-}t^{r-}$, then it is easy to check that $g(R)>r$ for some $R,r>0$. This proves (1).

Now (2) is immediate as for a fixed $w\in W_0^{1,p(x)}(\Omega)$ we have
$$
\lim_{t\to \infty}\mathcal{F}(tw) = -\infty.
$$

\medskip

Now the candidate for critical value according to the Mountain Pass Theorem is
$$
c = \inf_{g\in \mathcal{C}} \sup_{t\in[0,1]} \mathcal{F}(g(t)),
$$
where $\mathcal{C}=\{g:[0,1]\to W_0^{1,p(x)}(\Omega)\colon g \mbox{ continuous and } g(0)=0, g(1)=v_0\}$.

We will show that, if $\inf_{x\in\A_\delta}\lambda(x)$ is big enough for some $\delta>0$ then $c < \left(\frac{1}{p+} - \frac{1}{q^-_{\A}}\right) S^n$ and so the local Palais-Smale condition (Theorem \ref{Lemma.PS}) can be applied.

We fix $w\in W_0^{1,p(x)}(\Omega)$. Then, if $t<1$ we have
\begin{align*}
\mathcal{F}(tw) &\leq \int_\Omega t^{p(x)}\frac{|\nabla w|^{p(x)}}{p-} - t^{q(x)} \frac{|w|^{q(x)}}{q+} - \lambda(x) t^{r(x)} \frac{| w|^{r(x)}}{r+}\, dx\\
&\leq \frac{t^{p-}}{p-} \int_\Omega |\nabla w|^{p(x)}\, dx - \frac{t^{r+}}{r+}  \int_\Omega \lambda(x)|w|^{r(x)}\, dx\\
&\leq \frac{t^{p-}}{p-} \int_\Omega |\nabla w|^{p(x)}\, dx - \frac{t^{r+}}{r+}  \int_{\A_\delta} \lambda(x)|w|^{r(x)}\, dx\\
&\leq \frac{t^{p-}}{p-} \int_\Omega |\nabla w|^{p(x)}\, dx - \frac{t^{r+}}{r+}  \int_{\A_\delta}(\inf_{x\in \A_\delta} \lambda(x))|w|^{r(x)}\, dx
\end{align*}
We define $g(t) := \frac{t^{p-}}{p-} a_1 -(\inf_{x\in \A_\delta} \lambda(x))\frac{t^{r+}}{r+} a_3$, where $a_1$ and $a_2$ are given by $a_1 = \| |\nabla w|^{p(x)}\|_{L^1(\Omega)}$ and $a_3 = \| |w|^{r(x)}\|_{L^1(\A_\delta)}$.

The maximum of $g$ is attained at $t_\lambda = \left(\frac{a_1 }{(\inf_{x\in \A_\delta} \lambda(x)) a_3}\right)^\frac{1}{r+-p-}$. So, we conclude that there exists $\lambda_0>0$ such that if $(\inf_{x\in \A_\delta} \lambda(x))\geq\lambda_0$ then
$$
\mathcal{F}(tw) < \left(\frac{1}{p+} - \frac{1}{q^-_{\A}}\right) S^n
$$
This finishes the proof.
\end{proof}

\begin{remark}
Observe that if $\lambda(x)$ is continuous it suffices to assume that $\lambda(x)$ is large in the {\em criticality set} $\A$.
\end{remark}

\subsection{Proof of Theorem \ref{r<p}}
Now it remains to prove Theorem \ref{r<p}. So we begin by checking the Palais-Smale condition for this case.

\begin{lema}
Let $\{u_j\}_{j\in\N}\subset W^{1,p(x)}_0(\Omega)$ be a Palais-Smale sequence for $\mathcal{F}$ then $\{u_j\}_{j\in\N}$ is bounded.
\end{lema}

\begin{proof}
Let $\{u_j\}_{j\in\N}\subset W_0^{1,p(x)}(\Omega)$ be a Palais-Smale sequence, that is
$$
\mathcal{F}(u_j)\to c \quad\mbox{ and } \quad \mathcal{F}'(u_j)\to 0.
$$
Therefore there exists a sequence $\varepsilon_j \to 0$ such that
$$
|\mathcal{F}'(u_j)w|\leq \varepsilon_j \|w\|_{W_0^{1,p(x)}(\Omega)}\mbox{ for all }w\in W_0^{1,p(x)}(\Omega).
$$
Now we have,
\begin{align*}
c+1 &\geq \mathcal{F}(u_j)-\frac{1}{q^-}\mathcal{F}'(u_j)u_j + \frac{1}{q^-}\mathcal{F}'(u_j)u_j\\
&\geq \left(\frac{1}{p^+}-\frac{1}{q^-}\right) \int_\Omega |\nabla u_j|^{p(x)}\, dx + \int_\Omega \left(\frac{\lambda(x)}{q^-}-\frac{\lambda(x)}{r^-}\right) |u_j|^{r(x)}\, dx + \frac{1}{q^-} \mathcal{F}'(u_j)u_j
\end{align*}
We can assume that $\| |\nabla u_j| \|_{L^{p(x)}(\Omega)} > 1$. Then we have, by Proposition \ref{norma.y.rho} and by Poincar\'e inequality,
\begin{align*}
c+1 \geq& \left(\frac{1}{p^+} - \frac{1}{q^-}\right) \| |\nabla u_j| \|^{p^-}_{L^{p(x)}(\Omega)} +\|\lambda\|_\infty \left(\frac{1}{q^-} - \frac{1}{r^-}\right) \|u_j\|^{r^+}_{L^{r(x)}(\Omega)} - \frac{1}{q^-}\|u_j\|_{W_0^{1,p(x)}(\Omega)}\varepsilon_j\\
\geq& \left(\frac{1}{p^+} - \frac{1}{q^-}\right) \| |\nabla u_j| \|^{p^-}_{L^{p(x)}(\Omega)} + \|\lambda\|_\infty\left(\frac{1}{q^-} - \frac{1}{r^-}\right) C \| |\nabla u_j| \|^{r^+}_{L^{p(x)}(\Omega)} - \frac{1}{q^-}\|u_j\|_{W_0^{1,p(x)}(\Omega)}
\end{align*}
from where it follows that $\|u_j\|_{W_0^{1,p(x)}(\Omega)}$ is bounded (recall that $p^+\leq q^-$ and $r^+<p^-$).
\end{proof}

Let $\{u_j\}_{j\in\N}$ be a Palais-Smale sequence for $\mathcal{F}$. Therefore, by the previous Lemma, it follows that $\{u_j\}_{j\in\N}$ is bounded in $W^{1,p(x)}_0(\Omega)$.

Then, by Theorem \ref{ccp} we can assume that there exist two measures $\mu, \nu$ and a function $u\in W_0^{1,p(x)}(\Omega)$ such that
\begin{align}
& u_j\rightharpoonup u \mbox{ weakly in } W^{1,p(x)}_0(\Omega),\\
& |\nabla u_j|^{p(x)} \rightharpoonup \mu \mbox{ weakly in the sense of measures},\\
& |u_j|^{q(x)} \rightharpoonup \nu \mbox{ weakly in the sense of measures},\\
& \nu = |u|^{q(x)} + \sum_{i\in I} \nu_i \delta_{x_i},\\
& \mu \ge |\nabla u|^{p(x)} + \sum_{i\in I} \mu_i \delta_{x_i},\\
& S \nu_i^{1/p^*(x_i)} \le \mu_i^{1/p(x_i)}.
\end{align}
As before, assume that $I\not= \emptyset$. Now the proof follows exactly as in the previous case, until we get to
\begin{align*}
c\geq& \left(\frac{1}{p^+}-\frac{1}{q^-}\right) \int_{\Omega}|u|^{q(x)}\, dx + \left(\frac{1}{p^+} - \frac{1}{q^-}\right) S^n + \|\lambda\|_{L^\infty(\Omega)}\left(\frac{1}{p^+} - \frac{1}{r^-}\right) \int_\Omega|u|^{r(x)}\, dx.
\end{align*}
Applying now H\"older inequality, we find
\begin{align*}
c \geq& \left(\frac{1}{p^+}-\frac{1}{q^-}\right) \int_{\Omega} |u|^{q(x)}\, dx + \left(\frac{1}{p^+}-\frac{1}{q^-}\right) S^n\\
&+ \|\lambda\|_{L^\infty(\Omega)} \left(\frac{1}{p^+}-\frac{1}{r^-}\right) \| |u|^{r(x)}\|_{L^{q(x)/r(x)}(\Omega)} |\Omega|^\frac{q^+}{q^- -r^+}.
\end{align*}
If $\| |u|^{r(x)} \|_{L^{q(x)/r(x)}(\Omega)}\geq 1$, we have
$$
c\geq c_1\| |u|^{r(x)} \|^{(q/r)^-}_{L^{q(x)/r(x)}(\Omega)} + c_3 - \|\lambda\|_{L^\infty(\Omega)} c_2\| |u|^{r(x)} \|_{L^{q(x)/r(x)}(\Omega)},
$$
so, if $f_1(x):=c_1 x^{(q/r)^-} - \|\lambda\|_{L^\infty(\Omega)} c_2 x$, this function reaches its absolute minimum at  $x_0 = \left(\frac{\|\lambda\|_{L^\infty(\Omega)} c_2}{c_1 (q/r)^-}\right)^\frac{1}{(q/r)^- -1}$.

On the other hand, if $\| |u|^{r(x)} \|_{L^{q(x)/r(x)}(\Omega)} < 1$, then
$$
c \geq c_1\| |u|^{r(x)} \|^{(q/r)^+}_{L^{q(x)/r(x)}(\Omega)} + c_3 - \|\lambda\|_{L^\infty(\Omega)} c_2 \|u\|_{L^{q(x)/r(x)}(\Omega)},
$$
so, if $f_2(x)=c_1 x^{(q/r)^+} - \|\lambda\|_{L^\infty(\Omega)} c_2 x$, this function reaches its absolute minimum at $x_0 = \left(\frac{\|\lambda\|_{L^\infty(\Omega)} c_2}{c_1 (q/r)^+}\right)^\frac{1}{(q/r)^+ -1}$.

Then, we obtain
\begin{align*}
c\geq\left(\frac{1}{p+}-\frac{1}{q^-}\right)S^n + K
\min \left\{\|\lambda\|_{L^\infty(\Omega)}^{\frac{(q/r)^-}{(q/r)^- -1}}, \|\lambda\|_{L^\infty(\Omega)}^{\frac{(q/r)^+}{(q/r)^+ -1}}\right\},
\end{align*}
which contradicts our hypothesis.

Therefore $I=\emptyset$ and so $u_j\to u$ strongly in $L^{q(x)}(\Omega)$.

With these preliminaries the Palais-Smale condition can now be easily checked.
\begin{lema}\label{elijo c}
Let $(u_j)\subset W^{1,p(x)}_0(\Omega)$ be a Palais-Smale sequence for $\mathcal{F}$, with energy level $c$. There exists a constant $K$ depending only on $p,q,r$ and $\Omega$ such that, if $c<\left(\frac{1}{p+}-\frac{1}{q^-}\right)S^n + K
\min \left\{\|\lambda\|_{L^\infty(\Omega)}^{\frac{(q/r)^-}{(q/r)^- -1}}, \|\lambda\|_{L^\infty(\Omega)}^{\frac{(q/r)^+}{(q/r)^+ -1}}\right\}$, then there exists a subsequence $\{u_{j_k}\}_{k\in\N}\subset \{u_j\}_{j\in\N}$ that converges strongly in $W^{1,p(x)}_0(\Omega)$.
\end{lema}

\begin{proof}
At this point, the proof follows by the continuity of the solution operator as in Theorem \ref{Lemma.PS}.
\end{proof}

Assume now that $\| |\nabla u| \|_{L^{p(x)}(\Omega)}\leq 1$. Then, applying Poincar\'e inequality, we have
\begin{align*}
\mathcal{F}(u) &\geq \frac{1}{p^+} \| |\nabla u| \|^{p^+}_{L^{p(x)}(\Omega)} - \frac{1}{q^-}\|u\|^{q^-}_{L^{q(x)}(\Omega)} - \frac{\|\lambda\|_{L^\infty(\Omega)}}{r^-}\|u\|_{L^{r(x)}(\Omega)}^{r^-}\\
&\geq \frac{1}{p^+}\| |\nabla u| \|^{p^+}_{L^{p(x)}(\Omega)} - \frac{C}{q^-} \| |\nabla u| \|^{q^-}_{L^{p(x)}(\Omega)} - \frac{\|\lambda\|_{L^\infty(\Omega)} C}{r^-} \| |\nabla u| \|_{L^{p(x)}(\Omega)}^{r^-} =: J_1(\| |\nabla u| \|_{L^{p(x)}(\Omega)}),
\end{align*}
where $J_1(x)=\frac{1}{p^+}x^{p^+}-\frac{C}{q^-}x^{q^-}-\frac{\|\lambda\|_{L^\infty(\Omega)} C}{r^-}x^{r^-}$. We recall that $p^+\leq q^-$ and $r^-<r^+<p^-<p^+$.

As $J_1$ attains a local, but not a global, minimum ($J_1$ is not  bounded below), we have to perform some sort of truncation. To this end let $x_0, x_1$ be such that $m<x_0<M<x_1$ where $m$ is the local minimum and $M$ is the local maximum of $J_1$ and $J_1(x_1)>J_1(m)$. For these values $x_0$ and $x_1$ we can choose a smooth function  $\tau_1(x)$ such that $\tau_1(x)=1$ if $x\leq x_0$, $\tau_1(x)=0$ if $x\geq x_1$ and $0\leq\tau_1(x)\leq 1$.

If $\| |\nabla u| \|_{L^{p(x)}(\Omega)}>1$, we argue similarly and obtain
\begin{align*}
\mathcal{F}(u)&\geq\frac{1}{p^+}\| |\nabla u| \|^{p^-}_{L^{p(x)}(\Omega)} - \frac{C}{q^-} \| |\nabla u| \|^{q^+}_{L^{p(x)}(\Omega)} - \frac{\|\lambda\|_{L^\infty(\Omega)} C}{r^-} \| |\nabla u| \|_{L^{p(x)}(\Omega)}^{r^+} =: J_2(\| |\nabla u| \|_{L^{p(x)}(\Omega)})
\end{align*}
where $J_2(x)=\frac{1}{p^+}x^{p^-}-\frac{C}{q^-}x^{q^+}-\frac{\|\lambda\|_{L^\infty(\Omega)} C}{r^-}x^{r^+}$. As in the previous case, $J_2$ attains a local but not a global minimum. So let $x_0,x_1$ be such that $m<x_0<M<x_1$ where $m$ is the local minimum of $j$ and $M$ is the local maximum of $J_2$  and $J_2(x_1)>J_2(m)$. For these values $x_0$ and $x_1$ we can choose a smooth function $\tau_2(x)$  with the same properties as $\tau_1$.
Finally, we define
$$
\tau(x)=\begin{cases}
\tau_1(x) & \mbox{if } x\leq 1\\
\tau_2 (x) &\mbox{if } x>1.
\end{cases}
$$
Next, let $\varphi(u) = \tau(\| |\nabla u| \|_{L^{p(x)}(\Omega)})$ and define the truncated functional as follows
$$
\tilde{\mathcal{F}}(u) = \int_\Omega \frac{|\nabla u|^{p(x)}}{p(x)}\, dx - \int_\Omega \frac{|u|^{q(x)}}{q(x)}\varphi(u)\, dx - \int_\Omega\frac{\lambda(x)}{r(x)}|u|^{r(x)}\, dx
$$

Now we state a Lemma that contains the main properties of $\tilde{\mathcal{F}}$.
\begin{lema}
$\tilde{\mathcal{F}}$ is $C^1$, if $\tilde{\mathcal{F}}(u)\leq 0$ then $\|u\|_{W_0^{1,p(x)}(\Omega)}<x_0$ and $\mathcal{F}(v)=\tilde{\mathcal{F}}(v)$ for every $v$ close enough to $u$. Moreover there exists $\lambda_1>0$ such that if $0<\|\lambda\|_{L^{\infty}(\Omega)} < \lambda_1$ then $\tilde{\mathcal{F}}$ satisfies a local Palais-Smale condition for $c\leq 0$.
\end{lema}
\begin{proof}
We only have to check the local Palais-Smale condition. Observe that every Palais-Smale sequence  for $\tilde{\mathcal{F}}$ with energy level $c\leq 0$ must be bounded, therefore by Lemma \ref{elijo c} if $\lambda$ verifies
$0<\left(\frac{1}{p+}-\frac{1}{q^-}\right)S^n + K
\min \left\{\|\lambda\|_{L^\infty(\Omega)}^{\frac{(q/r)^-}{(q/r)^- -1}}, \|\lambda\|_{L^\infty(\Omega)}^{\frac{(q/r)^+}{(q/r)^+ -1}}\right\}$, then there exists a convergent subsequence
\end{proof}
The following Lemma gives the final ingredients needed in the proof.
\begin{lema}\label{genero}
For every $n\in \N$ there exists $\varepsilon>0$ such that
$$
\gamma(\mathcal{\tilde{F}}^{-\varepsilon})\geq n
$$
where $\mathcal{\tilde{F}}^{-\varepsilon}=\{u\in W^{1,p(x)}_0(\Omega)\colon  \mathcal{\tilde{F}}(u)\leq -\varepsilon\}$ and $\gamma$ is the Krasnoselskii genus.
\end{lema}
\begin{proof}
Let $E_n\subset W_0^{1,p(x)}(\Omega)$ be a $n$-dimensional subspace. Hence we have, for $u\in E_n$ such that $\|u\|_{W_0^{1,p(x)}(\Omega)}=1$,
\begin{align*}
\mathcal{\tilde{F}}(tu)&= \int_\Omega \frac{|\nabla (tu)|^{p(x)}}{p(x)}\, dx - \int_\Omega \frac{|tu|^{q(x)}}{q(x)}\varphi(tu)\, dx - \int_\Omega \frac{\lambda(x)}{r(x)}|tu|^{r(x)}\, dx\\
&\le \int_\Omega \frac{|\nabla (tu)|^{p(x)}}{p^-}\, dx - \int_\Omega \frac{|tu|^{q(x)}}{q^+}\varphi(tu)\, dx - \int_{\Omega}\frac{\lambda(x)}{r^+}|tu|^{r(x)}\, dx.
\end{align*}
If $t<1$, then
\begin{align*}
\tilde{\mathcal{F}}(tu) &\leq \int_\Omega\frac{t^{p^-}|\nabla u|^{p(x)}}{p^-}\, dx - \int_\Omega\frac{t^{q^+}|u|^{q(x)}}{q^+}\, dx - \int_{\Omega} \frac{\inf_{x \in\Omega}\lambda(x)}{r^+}t^{r^+}|u|^{r(x)}\, dx\\
&\leq \frac{t^{p^-}}{p^-} - \frac{t^{q^+}}{q^+}a_n -\inf_{x \in\Omega}\lambda(x)\frac{t^{r^+}}{r^+}b_n,
\end{align*}
where 
$$
a_n=\inf\Big\{\int_\Omega|u|^{q(x)}\,dx\colon u\in E_n,\|u\|_{W_0^{1,p(x)}(\Omega)}=1\Big\}
$$
and
$$
b_n=\inf\Big\{\int_{\Omega}|u|^{r(x)}\,dx\colon u\in E_n,\|u\|_{W_0^{1,p(x)}(\Omega)}=1\Big\}.
$$
Now,we have
$$
\tilde{\mathcal{F}}(tu)\leq\frac{t^{p^-}}{p^-}-\frac{t^{q^+}}{q^+}a_n-\inf_{x\in\Omega}\lambda(x)\frac{t^{r^+}}{r^+}b_n\leq\frac{t^{p^-}}{p^-}-\inf_{x \in\Omega}\lambda(x)\frac{t^{r^+}}{r^+}b_n
$$
Observe that $a_n>0$ and $b_n>0$ because $E_n$ is finite dimensional. As $r^+<p^-$ and $t<1$ we obtain that  there exists positive constants $\rho$ and $\varepsilon$ such that
$$
\tilde{\mathcal{F}}(\rho u)<-\varepsilon \quad\mbox{ for }u\in E_n, \|u\|_{W_0^{1,p(x)}(\Omega)}=1.
$$
Therefore, if we set $S_{\rho,n}=\{u\in E_n:\|u\|=\rho\}$, we have that $S_{p,n}\subset\tilde{\mathcal{F}}^{-\varepsilon}$. Hence by monotonicity of the genus
$$
\gamma(\tilde{\mathcal{F}}^{-\varepsilon})\geq\gamma(S_{\rho,n})=n
$$
as we wanted to show.
\end{proof}

\begin{teo}
Let
\begin{align*}
&\Sigma=\{A\subset W^{1,p(x)}_0(\Omega)-{0}\colon A \mbox{ is closed},\ A=-A\},\\
&\Sigma_k=\{A\subset\Sigma\colon \gamma(A)\geq k\},
\end{align*}
where $\gamma$ stands for the Krasnoselskii genus.
Then
$$
c_k=\inf_{A\in\Sigma_k}\sup_{u\in A}\mathcal{F}(u)
$$
is a negative critical value of $\mathcal{F}$ and moreover, if $c=c_k=\cdots=c_{k+r}$, then $\gamma(K_c)\geq r+1$, where $K_c = \{u\in W^{1,p(x)}(\Omega)\colon \mathcal{F}(u)=c, \mathcal{F}'(u)=0\}$.
\end{teo}
\begin{proof}
The proof now follows exactly as in that of \cite{GAP} using Lemma \ref{genero}.
\end{proof}


\begin{thebibliography}{10}

\bibitem{Alves}
Claudianor~O. Alves.
\newblock Existence of positive solutions for a problem with lack of
  compactness involving the {$p$}-{L}aplacian.
\newblock {\em Nonlinear Anal.}, 51(7):1187--1206, 2002.

\bibitem{Alves-Ding}
Claudianor~O. Alves and Yanheng Ding.
\newblock Existence, multiplicity and concentration of positive solutions for a
  class of quasilinear problems.
\newblock {\em Topol. Methods Nonlinear Anal.}, 29(2):265--278, 2007.

\bibitem{Bahri-Lions}
Abbas Bahri and Pierre-Louis Lions.
\newblock On the existence of a positive solution of semilinear elliptic
  equations in unbounded domains.
\newblock {\em Ann. Inst. H. Poincar\'e Anal. Non Lin\'eaire}, 14(3):365--413,
  1997.

\bibitem{FB}
Juli{\'a}n~Fern{\'a}ndez Bonder, Sandra Mart{\'{\i}}nez, and Julio~D. Rossi.
\newblock Existence results for gradient elliptic systems with nonlinear
  boundary conditions.
\newblock {\em NoDEA Nonlinear Differential Equations Appl.}, 14(1-2):153--179,
  2007.

\bibitem{Cabada-Pouso}
Alberto Cabada and Rodrigo~L. Pouso.
\newblock Existence theory for functional {$p$}-{L}aplacian equations with
  variable exponents.
\newblock {\em Nonlinear Anal.}, 52(2):557--572, 2003.

\bibitem{Dinu}
Teodora-Liliana Dinu.
\newblock Nonlinear eigenvalue problems in {S}obolev spaces with variable
  exponent.
\newblock {\em J. Funct. Spaces Appl.}, 4(3):225--242, 2006.

\bibitem{Drabek-Huang}
Pavel Dr{\'a}bek and Yin~Xi Huang.
\newblock Multiplicity of positive solutions for some quasilinear elliptic
  equation in {${\bf R}\sp N$} with critical {S}obolev exponent.
\newblock {\em J. Differential Equations}, 140(1):106--132, 1997.

\bibitem{FZ}
Xian-Ling Fan and Qi-Hu Zhang.
\newblock Existence of solutions for {$p(x)$}-{L}aplacian {D}irichlet problem.
\newblock {\em Nonlinear Anal.}, 52(8):1843--1852, 2003.

\bibitem{Fan}
Xianling Fan and Dun Zhao.
\newblock On the spaces {$L\sp {p(x)}(\Omega)$} and {$W\sp {m,p(x)}(\Omega)$}.
\newblock {\em J. Math. Anal. Appl.}, 263(2):424--446, 2001.

\bibitem{Fu}
Yongqiang Fu.
\newblock The principle of concentration compactness in {$L\sp {p(x)}(\Omega)$}
  spaces and its application.
\newblock {\em Nonlinear Anal.}, 71(5-6):1876--1892, 2009.

\bibitem{GAP}
J.~Garc{\'{\i}}a~Azorero and I.~Peral~Alonso.
\newblock Multiplicity of solutions for elliptic problems with critical
  exponent or with a nonsymmetric term.
\newblock {\em Trans. Amer. Math. Soc.}, 323(2):877--895, 1991.

\bibitem{Lions}
P.-L. Lions.
\newblock The concentration-compactness principle in the calculus of
  variations. {T}he limit case. {I}.
\newblock {\em Rev. Mat. Iberoamericana}, 1(1):145--201, 1985.

\bibitem{Mihailescu}
Mihai Mih{\u{a}}ilescu.
\newblock Elliptic problems in variable exponent spaces.
\newblock {\em Bull. Austral. Math. Soc.}, 74(2):197--206, 2006.

\bibitem{Mihailescu-Radulescu}
Mihai Mih{\u{a}}ilescu and Vicen{\c{t}}iu R{\u{a}}dulescu.
\newblock On a nonhomogeneous quasilinear eigenvalue problem in {S}obolev
  spaces with variable exponent.
\newblock {\em Proc. Amer. Math. Soc.}, 135(9):2929--2937 (electronic), 2007.

\end{thebibliography}
\end{document}